\documentclass{amsart}
\usepackage{amssymb}
\usepackage{amsmath}
\usepackage{amsthm}
\usepackage{enumerate}
\usepackage{graphicx}
\usepackage{color}

\def\p{\varphi}

\def\l{\lambda}
\def\L{\Lambda}

\def\a{\alpha}
\def\b{\beta}
\def\d{\delta}

\def\pa{\partial}

\def\R{{\mathbb R}}

\def\A{\mathcal{A}}
\def\B{\mathcal{B}}

\def\L{\mathcal{L}}

\newtheorem{theorem}{Theorem}[section]
\newtheorem{lemma}[theorem]{Lemma}

\newtheorem{definition}[theorem]{Definition}
\newtheorem{example}[theorem]{Example}

\newtheorem{remark}{\bf Remark}[section]

\newfont{\Bb}{msbm10 scaled\magstep{1}}

\begin{document}

\title[Cracked beams and arches]{Equations of motion for cracked beams and shallow arches}

\author{Semion Gutman, Junhong Ha and Sudeok Shon}
\address{$^{1}$ Department of Mathematics, University of Oklahoma, Norman, Oklahoma 73019, USA, e-mail: sgutman@ou.edu}
\address{$^{1}$School of Liberal Arts, Korea University of Technology and Education,
        Cheonan 31253, South Korea, e-mail: hjh@koreatech.ac.kr}
\address{$^{3}$Department of Architectural Engineering, Korea University of Technology and Education,
        Cheonan 31253, South Korea, e-mail: sdshon@koreatech.ac.kr}



\begin{abstract} 
Cracks in beams and shallow arches are modeled by  massless rotational springs. First, we introduce a specially designed linear operator that "absorbs" the boundary conditions at the cracks. Then the  equations of motion are derived from the first principles using the Extended Hamilton's Principle, accounting for non-conservative forces. The variational formulation of the equations is stated in terms of the subdifferentials of the bending and axial potential energies. The equations are given in their abstract (weak), as well as in classical forms.
\end{abstract}

\maketitle


\section{Introduction}\label{section:intro}

Modeling dynamic behavior of cracked beams and arches has important engineering applications.
The main goal of this paper is to develop  a rigorous mathematical framework for such problems.

The theory of uniform beams and shallow arches is well developed. An early exposition can be found in  \cite{BALL197361}. More general models in the multidimensional setting, and a literature survey are presented in \cite{Emmrich_2011}. A review for vibrating beams is given in \cite{HAN1999}. Motion of uniform arches and a related parameter estimation problem are studied in \cite{GUTMAN2013297}. These results are extended to point loads in \cite{GH2017}. The existence of a compact, uniform attractor is established in \cite{GUTMAN2018557}. 

For a theory of cracked Bernoulli-Euler beams see  \cite{Christides1984}.  
A significant effort has been directed at the vibration analysis of cracked beams. Representation of a crack by a rotational spring has been proven to be accurate, and it is often used, see  \cite{CADDEMI2009, CADDEMI2013944} and the extensive bibliography there. Determination of the beam natural frequencies is discussed in  \cite{LIN2002987, OSTACHOWICZ1991191, SHIFRIN1999}. 
S. Caddemi and his colleagues have further developed the theory using energy functions in \cite{CADDEMI2013944}. Substantial reviews of cracked elements are presented in \cite{Cannizzaro2017, DIMAROGONAS1996, DIMAROGONAS1998}.

 The transverse motion of a beam or an  arch is described by the function $y(x,t),\, x\in [0,\pi],\, t\geq 0$, which represents the deformation of the beam/arch measured from the $x$-axis. For definiteness, the boundary conditions are of the hinged type
 \begin{equation}\label{intro:eq32}
y(0,t)=y''(0,t)=0,\quad y(\pi,t)=y''(\pi,t)=0,\quad t \in (0,T).
\end{equation}
 Other types of boundary conditions, can be treated similarly.
 
\begin{figure}
\begin{center}
\includegraphics[height=3.5cm,width=11.5cm]{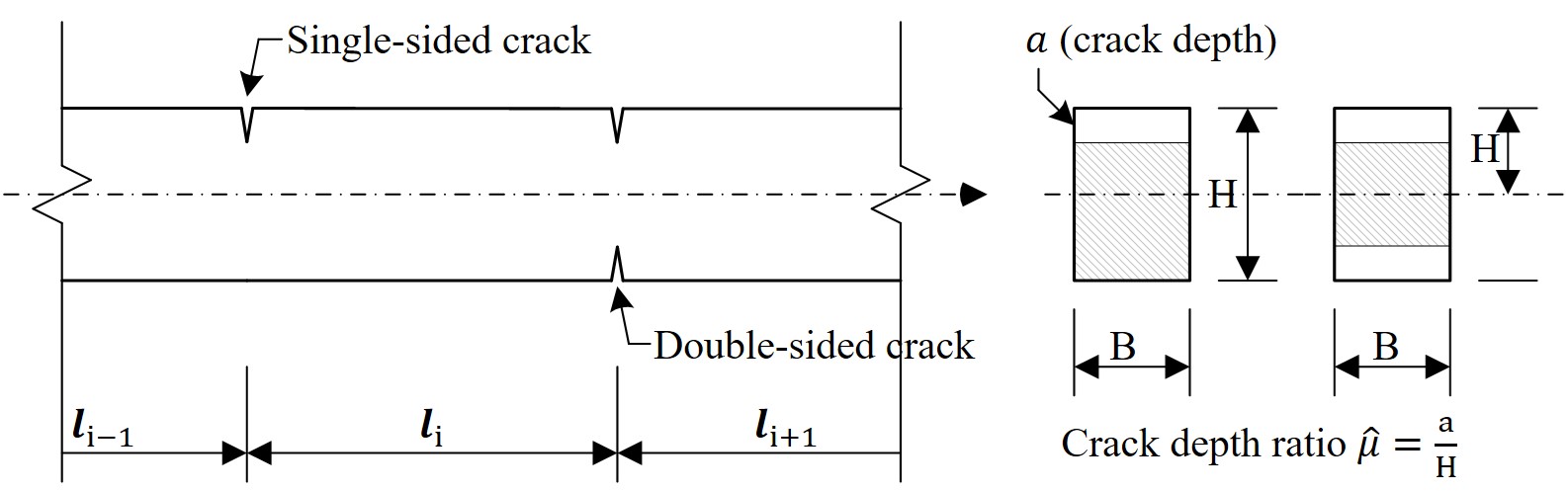}
\end{center}
\vspace{-0.5cm}
\caption{Crack parameters.} \label{figCracks}
\end{figure}
  A crack is fully described by its position along the axis, and the crack depth ratio $\hat\mu$, as shown in Figure \ref{figCracks}. 
According to the common practice in the field, see \cite{CERRI200439}, a crack is modeled by a massless rotational spring.
The  spring flexibility $\theta=\theta(\hat\mu)$ depends on the crack depth ratio $\hat\mu$, and on whether the crack is one-sided or two-sided, open or closed, and so on. The flexibility $\theta$ is equal to $0$ if there is no crack, and it increases with the crack depth.
Explicit expressions for the functions $\theta(\hat\mu)$ are provided in Section \ref{section:phys}.  

\begin{remark}
The following discussion is applicable to both arches and beams, but to avoid repetitions we will refer just to arches. 
\end{remark}

Suppose that there are $m$ cracks along the length of the arch, located at $0<x_1<\dots <x_m<\pi$.  For convenience, we denote $x_0=0$, and  $x_{m+1}=\pi$. 
Consequently, the cracked arch is modeled as a collection of $m+1$ uniform arches over the intervals $l_i=(x_{i-1},x_i),\, i=1,\dots, m+1$, as shown in Figure \ref{figCrArch}(b).

We consider only the transverse motion of the arch, so its position can be described by the function $y=y(x,t)$, $0\leq x\leq \pi$, $t\geq 0$.
The boundary conditions at the cracks enforce the continuity of the displacement field $y$, the bending moment $y''$, and the shear force $y'''$. 
Condition $y'(x_i^+,t)-y'(x_i^-,t)=\theta_i y''(x_i^+,t)$ expresses the discontinuity of the arch slope at the $i$-th crack, where $\theta_i=\theta(\hat\mu_i)$, see Figure \ref{figCrArch}(b). 

\begin{figure}[b]
\begin{center}
\includegraphics[height=2.5cm,width=11.5cm]{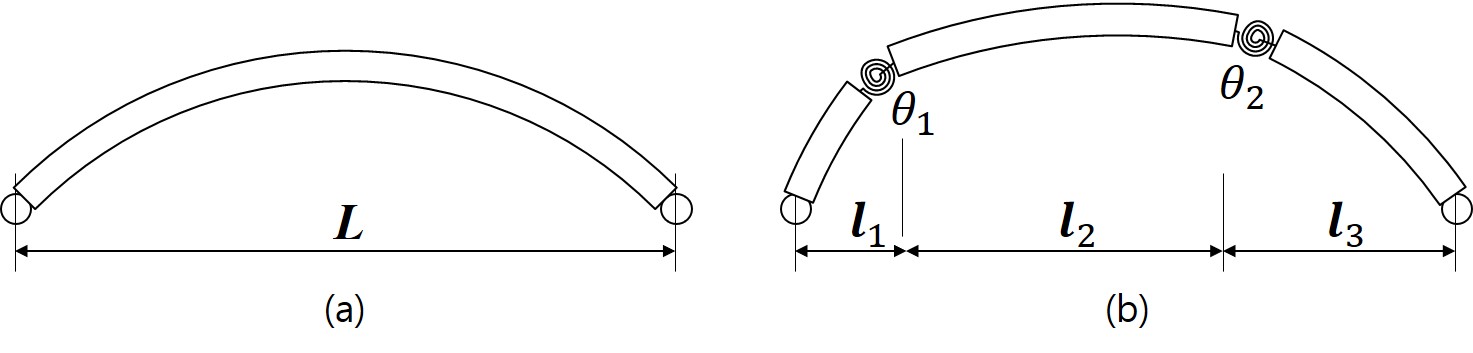}
\end{center}
\vspace{-0.5cm}
\caption{Beam or shallow arch: (a) uniform, (b) with two cracks.} \label{figCrArch}
\end{figure}

To simplify the statement of the boundary conditions at the cracks, we introduce the notion of the \emph{jump} $J[u](x)$ of a function $u=u(x)$ at any $x\in(0,\pi)$, as follows
\begin{equation}\label{intro:eq16}
 J[u](x)=u(x^+)-u(x^-).
\end{equation}

With this notation the conditions at the cracks (joint conditions) are
\begin{equation}\label{intro:eq18}
 J[y](x_i,t)=0,\quad J[y''](x_i,t)=0,\quad J[y'''](x_i,t)=0,
\end{equation}
and
\begin{equation}\label{intro:eq20}
 J[y'](x_i,t)=\theta_i y''(x_i^+,t),
\end{equation}
where $\theta_i=\theta(\hat\mu_i)$, $i=1,2,\dots,m$ and $t\geq 0$. Note that $y''(x_i^+,t)=y''(x_i^-,t)$ by \eqref{intro:eq18}.

In Section \ref{section:hilbert} we review our recent results from \cite{g37} on the variational setting for cracked beams and arches. 
First, special Hilbert spaces $V, H_0^1, H$ are defined satisfying
\begin{equation}\label{intro:eq36}
V\subset H_0^1\subset H\subset (H_0^1)'\subset V',
\end{equation}
with continuous and dense embeddings. These spaces are broad enough to contain continuous functions with discontinuous derivatives at the joint points.

Then we introduce the operator $\A : V\to V'$, by

\begin{equation}\label{intro:eq38}
\langle \A u,v\rangle_V= \sum_{i=1}^{m+1} (u'', v'')_i+\sum_{i=1}^m \frac 1\theta_i J[u'](x_i) J[v'](x_i), 
\end{equation}
for any $u,v\in V$, where $J[u'](x)=u'(x^+)-u'(x^-)$.

Our main result in \cite{g37} is that the solution $u$ of the equation $\A u=f$ in $H$ satisfies the joint conditions, including \eqref{intro:eq20}. Thus the operator $\A$ "absorbs" the boundary conditions, as expected of the weak  formulation of the steady state problem.
This result allows us to mathematically justify the existence of the eigenvalues and the eigenfunctions of $\A$.

 In Section \ref{section:phys} we consider relevant physical quantities, including the potential energy $U_b$, due to bending, and the potential  energy $U_a$, due to the axial force. This is the only section in the paper, that contains physical variables. Their non-dimensional equivalents are used in all the other sections.

Typically,  equations of motion are derived from the Hamilton's Principle $\d I=0$, which seeks the stationary paths of the action $I$. A closer examination of this statement in the framework of the Hilbert spaces reveals the importance of the subdifferential $\pa\phi$ of a convex lower-semicontinuous function $\phi$. The potential energies $U_b$, and $U_a$ are two examples of such functions.

Section \ref{section:conv} presents a brief review of these concepts.  One can view the subdifferential, which can be multi-valued,  as a generalization of the derivative. For example, let $\phi(x)=|x|,\, x\in \R$. Then $\pa\phi(x)=-1$, for $x<0$, $\pa\phi(x)=1$, for $x>0$, and $\pa\phi(0)=[-1,1]$. Geometrically, it means that the "tangent" line to the graph of $\phi$ at $x=0$ can have any slope between $-1$ and $1$. We conclude Section \ref{section:conv} by deriving various expressions for the subdifferentials $\pa U_b$, and $\pa U_a$. In particular, we show that $\pa U_b=\A$.

Section \ref{section:ham} uses the Extended Hamilton's Principle, which accommodates non-conservative forces, to derive the equations of motion. It contains our main result: the abstract equation of motion for cracked beams and arches is
\begin{equation}\label{intro:eq40}
\ddot y+\pa U_b(y)+\pa U_a(y)+c_d\dot y=p,
\end{equation}
where $\dot y, \ddot y$ denote the time derivatives.

For beams, it is assumed that the influence of the axial force can be neglected. This case is considered in Section \ref{section:beameq}. The abstract equation for  cracked beams is
\begin{equation}\label{intro:eq42}
\ddot y+\A y+c_d\dot y=p.
\end{equation}
If the beam contains no cracks, then \eqref{intro:eq42} becomes $\ddot y+y''''+c_d\dot y=p$,
which is consistent with the classical Euler-Bernoulli theory.

Our main result in Section \ref{section:archeq} is that the "classical" equation for a cracked shallow arch is  
\begin{align*}
&\ddot y+ y''''\\
&-\frac 1\pi\left(\b+ \frac{1}{2}\int_0^L|y'(x,t)|^2\, dx\right)\left(y''+\sum_{i=1}^m \theta_i y''(x_i,t)\d(x-x_i)\right)+ c_d\dot y=p,
\end{align*}
where $\d=\d(x)$ is the delta function. 

Motion in viscous media results in the additional term $\mu\A\dot y$, $\mu>0$ in the governing equations. Such a case is referred to as the strong damping motion. If the viscous effects are neglected ($\mu=0$), we have the weak damping case.

\section{Variational setting for cracked beams and arches}\label{section:hilbert}
\setcounter{equation}{0}
This section contains a brief review of our results from \cite{g37}, to which we refer for further details. 
Let $H$ be the Hilbert space
\begin{equation}\label{hilbert:eq10}
H=\bigoplus_{i=1}^{m+1} L^2(l_i).
\end{equation}

Let the inner product and the norm in $L^2(l_i)$ be denoted by $(\cdot,\cdot)_i$ and $|\cdot|_i$ correspondingly. 
The inner product and the norm in $H$ are defined by
\begin{equation}\label{hilbert:eq12}
(u,v)_H=\sum_{i=1}^{m+1}(u,v)_i, \quad |u|^2_H= \sum_{i=1}^{m+1}|u|^2_i.
\end{equation}

Consider the Sobolev space $H^2(a,b)$ on a bounded interval $(a,b)\subset \R$, and let
$u\in H^2(a,b)$. Then $u, u'$ are continuous functions on $[a,b]$, up to a set of measure zero, and $u''\in L^2(a,b)$. Therefore, for such $u$, we will always assume that $u, u'\in C[a,b]$.

Define the linear space
\begin{equation}\label{hilbert:eq22}
V=\left\{ u\in \bigoplus_{i=1}^{m+1} H^2(l_i)\, :\, u(0)=u(\pi)=0,\ J[u](x_i)=0,\ i=1,\dots, m\right\}.
\end{equation}
We interpret  $u\in V$ as a continuous function on $[0,\pi]$, such that $u(0)=u(\pi)=0$, with $u'\in L^2(0,\pi)$, i.e. $u\in H_0^1(0,\pi)$. Furthermore,  $u|_{l_i}\in H^2(l_i)$, and $u'|_{l_i}\in C[x_{i-1}, x_i]$ for  $i=1,2,\dots, m+1$. 

 Define the inner product on $V$  by
\begin{equation}\label{hilbert:eq24}
((u,v))_V=\sum_{i=1}^{m+1} (u'', v'')_i+\sum_{i=1}^m J[u'](x_i) J[v'](x_i), \quad\text{for any}\ u,v\in V,
\end{equation}
where $(u'',v'')_i=\int_{l_i} u''(x) v''(x)\, dx$.

 The corresponding norm in $V$ is
\begin{equation}\label{hilbert:eq26}
\|u\|_V^2=\sum_{i=1}^{m+1} |u''|^2_i+\sum_{i=1}^m |J[u'](x_i)|^2, \quad\text{for any}\ u\in V,
\end{equation}
where $|\cdot|_i$ is the norm in $L^2(l_i)$. It can be shown that $V$ is a Hilbert space. 
 
 At this time we introduce the Hilbert space $H_0^1=H_0^1(0,\pi)$, with the inner product and the norm given by
 \begin{equation}\label{hilbert:eq60}
(u,v)_1=(u',v')_H,\quad \|u\|_1^2=|u'|_H^2, \quad u, v\in H_0^1.
\end{equation}
The norm in $(H_0^1)'$ will be denoted by $\|\cdot\|_{-1}$. It can be shown that the identity embedding $i : V\to H_0^1$ is linear, continuous, with a dense range in $H_0^1$. Furthermore, it is compact.
 
This allows us to define the Gelfand triple $V\subset H\subset V'$, with the pairing $\langle\cdot,\cdot\rangle_V$ between $V$ and $V'$ extending the inner product in $H$. This means that given $f\in H=H'\subset V'$, and $v\in V$, we have $\langle f,v\rangle_V=(f,v)_H$. 

Also, we have 
\begin{equation}\label{hilbert:eq62}
V\subset H_0^1\subset H\subset \left( H_0^1\right)'\subset V',
\end{equation}
with dense embeddings. Furthermore, the embeddings $ V\subset H_0^1\subset H$ are compact.


Now we can introduce the operator $\A : V\to V'$ that "absorbs" the junction boundary conditions.
This operator is central to the variational setting of problems for cracked beams and arches.

\begin{definition}\label{var:def2}
Define the  operator $\A$ on $V$ by
\begin{equation}\label{var:eq8}
\langle \A u,v\rangle_V= \sum_{i=1}^{m+1} (u'', v'')_i+\sum_{i=1}^m \frac 1\theta_i J[u'](x_i) J[v'](x_i), 
\end{equation}
for any $u,v\in V$. We will also write $\langle \A u,v\rangle$ for $\langle \A u,v\rangle_V$, if it does not cause a confusion.
\end{definition}
 Recall, that a linear operator $A : V\to V'$ is called coercive, if there exists $c>0$, such that $\langle Au, u\rangle\geq c\|u\|^2_V$ for any $u\in V$. We have
\begin{lemma}\label{var:lemma1}
Let $\A$ be defined by \eqref{var:eq8}. Then $\A$ is a symmetric, continuous, linear, and coercive operator from $V$ onto $V'$.
\end{lemma}

As was mentioned in Section \ref{section:intro}, functions $u=u(x)$ modeling an arch with cracks are expected to satisfy certain boundary conditions. For convenience, we restate them here:
\begin{equation}\label{var:eq2}
u(0)=u(\pi)=0,\quad u''(0)=u''(\pi)=0,
\end{equation}
and 
\begin{equation}\label{var:eq4}
 J[u](x_i)=0,\quad J[u''](x_i)=0,\quad J[u'''](x_i)=0,\quad J[u'](x_i)=\theta_i u''(x_i^+),
\end{equation}
for $i=1,\dots,m$.

The next theorem is the main result of this section.
\begin{theorem}\label{var:thm2}
Let the domain of $\A$ be $D(\A)=\{ v\in V : \A v\in H\}$.
\begin{enumerate}[(i)]
\item If $u\in D(\A)$, then $u|_{l_i}\in H^4(l_i)$, $\A u=u''''$ a.e. on $l_i$,  $i=1,\dots,m+1$, and $u$ satisfies  conditions \eqref{var:eq2}--\eqref{var:eq4}. 
\item Let $f\in H$, then equation $\A u=f$ in $V'$ has a unique solution $u\in D(\A)$.
\end{enumerate}
\end{theorem}

{\bf Remark}. The fact that $u''''=f$ a.e. on $(0,\pi)$ in Theorem \ref{var:thm2} does not imply that $u\in H^4(0,\pi)$. This is similar to the fact that the strong derivative $p'$ of a step function $p$ on $(0,\pi)$ is zero a.e. on $(0,\pi)$. However,  $p\not\in H^1(0,\pi)$. 

Finally in this section we discuss the {\bf eigenvalues} and the {\bf eigenfunctions} of the operator $\A$.
 It was shown in Lemma 
\ref{var:lemma1} that $\A$ is a continuous, linear, symmetric, and coercive operator from $V$ onto $V'$. 
Following \cite[Section 2.2.1]{temam1997infinite}, 
$\A$ can also be considered as an unbounded operator in $H$.

 Since the embedding $V\subset H$ is compact, the standard spectral theory for Sturm-Liouville boundary value problems is applicable. The eigenfunctions belong to $H$. Therefore, by Theorem \ref{var:thm2}, they are in the domain  $D(\A)\subset V$, thus continuous on $[0,\pi]$, and satisfy  conditions \eqref{var:eq2}--\eqref{var:eq4}.

 We summarize these results in the following lemma.
\begin{lemma}\label{var:lemma20}
Let $\A$ be the operator defined in \eqref{var:eq8}. Then
\begin{enumerate}[(i)]
\item There exists an increasing sequence of its real positive eigenvalues \newline
$\l^4_1,\l^4_2,\dots$, with $\lim_{k\to\infty}\l^4_k=\infty$. 
\item The corresponding eigenfunctions $\varphi_k\in D(\A)\subset V$, $k\geq 1$, and they satisfy the junction conditions \eqref{var:eq2}--\eqref{var:eq4}.
\item The eigenfunctions $\varphi_k$ satisfy $\A\varphi_k=\l^4_k\varphi_k$ in $H$, $k \geq 1$. That is, $\varphi_k''''(x)=\l^4_k\varphi_k(x)$ a.e. on every interval $l_i$, $i=1,\dots, m+1$. 
\item The set $\{\varphi_k\}_{k=1}^\infty$ is a complete orthonormal basis in $H$.

\end{enumerate} 
\end{lemma}

An efficient method  for a computational determination of the eigenvalues and the eigenfunctions of $\A$ (Modified Shifrin's method) is discussed in \cite{g37}.

\section{Physical parameters and their non-dimensional equivalents}\label{section:phys}
\setcounter{equation}{0} 
This is the only section in the paper 
 that uses physical variables. All the other sections use their non-dimensional equivalents. Since both the physical variables and their equivalents use the same notation, such an arrangement helps to avoid confusing them. The physical variables are contained in
Table \ref{phys:table1} together with their units. 
 Recall that the newton  $N=kg\cdot m\cdot s^{-2} $ is the unit of force.

\begin{table}[ht]
\begin{tabular}{llll}
$L$     & Length of bar & $m$ &			\\ 
$T$     & Final time & $s$ &    \\ 
$y$     & Bar deflection & $m$ & $y=y(x,t)$    \\ 
$p$     & Transverse load/unit of length & $N\cdot m^{-1}$ & $p=p(x,t)$    \\ 
$c_d$     & Damping coefficient & $kg\cdot m^{-1} s^{-1}   $ &    		\\
$\kappa$  & Bar curvature &  $m^{-1}$ &   			\\ 
$\rho$  & Volume mass density &  $kg\cdot m^{-3}$ &  		\\ 
$E$     & Young's modulus & $N\cdot m^{-2}$ &	\\ 
$A$     & Cross-sectional area & $m^2$ &    \\ 
$I$     & Area moment of inertia & $m^4$ &   		\\ 
$r$     & Radius of gyration & $m$ & $r=\sqrt{I/A}$    \\ 
$\omega_0$     & $t$-scale & $s^{-1}$ &     \\ 
$\omega$     & Natural frequency & $s^{-1}$ &     \\
$k$     & Spring stiffness &  $N\cdot m/rad$ &   				\\ 
$\theta$     & Flexibility & $m$ & $\theta=EI/k$    \\ 
$T_k$     & Kinetic energy &    $J=N\cdot m$ & 				\\
$U_a$     & Potential energy for axial force & $J$ &    \\
$U_b$     & Potential energy due to bending &  $J$ &   \\
$M$         & Bending moment magnitude &   $N\cdot m$	 &	\\
$S_0$         & Initial tensile axial force & $N$ &   \\
$S_1$         & Tensile force due to deflection &  $N$ & \\
$N$         & Total axial force & $N$ & $N=S_0+S_1$  \\
$\L$         & Lagrangian &$J$ & $\L=T_k-U_b-U_a$  \\ 
$\mathcal{W}_{ext}$       & External work &  $J$ & \\
$\mathcal{W}_{d}$     & Dissipative work & $J$ &    \\ 
$\mathcal{W}_{nc}$   & Non-conservative work & $J$  & $\mathcal{W}_{nc}=\mathcal{W}_{d}+\mathcal{W}_{ext}$  \\
$e_I$         & Energy normalization factor &$J$ &   \\ 
$\b$         & Axial force renormalization & $1$  \\
$\hat\mu$         & Crack depth ratio & $1$  \\
$\mu$         & Normalized dynamic viscosity & $1$  \\
\end{tabular}
\vskip2mm
\caption{Nomenclature} \label{phys:table1}
\end{table}

To use the Extended Hamilton's Principle we need to derive  the 
Lagrangian $\mathcal{L}=T_k-U=T_k-U_b-U_a$, where $U_b, U_a$ and $T_k$ are the potential and the kinetic energies. Also, we need the non-conservative external work $\mathcal{W}_{nc}$, which is the sum of the external work $\mathcal{W}_{ext}$ due to the load $p$, and the dissipative work $\mathcal{W}_{d}$ due to the damping force $\mathcal{F}_d$. 

Let $y(x,t)$, $0\leq x\leq L$,  be the position of the arch at the time $t\geq 0$. As usual, the dots denote the time derivatives, and the primes denote the spatial ones.
For the kinetic energy we have
\begin{equation}\label{phys:eq4}
T_k(y)=\frac{\rho A}{2}\int_0^L |\dot y(x,t)|^2\, dx.
\end{equation}
The external work $\mathcal{W}_{ext}$ by the non-conservative load $p(x,t)$ is 
\begin{equation}\label{phys:eq8}
\mathcal{W}_{ext}( y)=\int_0^L p(x,t)y(x,t)\, dx.
\end{equation}

The dissipative force $\mathcal{F}_d(x,t)=-c_d\dot y(x,t)$, $c_d\geq 0$ is a uniformly distributed viscous damping force acting only in the transverse direction. So, the 
dissipative work $\mathcal{W}_{d}$ done by the force is given by
\begin{equation}\label{phys:eq10}
\mathcal{W}_{d}(y)=\int_0^L\mathcal{F}_d(x,t) y(x,t)\, dx=-c_d\int_0^L \dot y(x,t)y(x,t)\, dx.
\end{equation}

The potential energy is 
\begin{equation}\label{phys:eq14}
U(y)= U_b(y)+U_a(y),
\end{equation}
where $U_b$ is the potential energy due to the bending, and $U_a$ is the potential energy due to the axial force. The particular form of these terms depends on the presence of cracks and other factors. 

To simplify the notations the $t$-dependency of $y$ is suppressed, if it does not cause a confusion.

{\bf Beam}. Suppose that we have a uniform beam with no cracks, as shown in Figure \ref{figCrArch}(a). 
The magnitude $M(x)$ of the bending moment vector $\vec M(x)$ of the beam at any point $x\in(0,L)$ is given by $M(x)=EI\kappa(x)$, where $\kappa(x)$ is the curvature of the beam at $x$. Assuming $|y'|\ll 1$, we get $\kappa(x)=y''(x)$, and  $M(x)=EI y''(x)$. Accordingly, the bending potential energy of the beam is given by
 \begin{equation}\label{phys:eq20}
U_b(y)=\frac { EI}2\int_0^L (y'')^2\, dx.
\end{equation}

Now suppose that the beam has cracks as in Section \ref{section:intro}, and in Figure \ref{figCrArch}(b). The cracks are at $0<x_1<\dots< x_m<L$.
The standard approach to modeling a crack is to represent it as a massless rotational spring with the spring constant $k$, and the flexibility $\theta$. 

The spring constant $k$ relates the torque ($N\cdot m$) to the angle of rotation. In our case this relationship takes the form $EI y''(x)=k J[y'](x)$, or $J[y'](x)=\theta y''(x)$, where  
\begin{equation}\label{phys:eq22}
\theta=\frac{EI}{k}.
\end{equation}
Thus the unit of the flexibility $\theta$ is $m$.

If the beam has a rectangular cross-section, as shown in Figure \ref{figCracks}, then the area moment of inertia $I$ of the rectangle can be computed explicitly, and \eqref{phys:eq22} can be simplified further. If the crack is double-sided, then by \cite[Eq. (2.8)-(2.10)]{OSTACHOWICZ1991191}, the expression for the flexibility $\theta$ becomes
\begin{equation}\label{phys:eq24}
\theta=6\pi H\hat\mu^2 (0.535-0.929\hat\mu+3.500\hat\mu^2-3.181\hat\mu^3+5.793\hat\mu^4),
\end{equation}
where $H$ is the half-height of the beam cross-section, and $\hat\mu=a/H$.

If the crack is single-sided, then by \cite[Eq. (2.8)-(2.10)]{OSTACHOWICZ1991191}
\begin{equation}\label{phys:eq26}
\theta=6\pi H\hat\mu^2 (0.6384-1.035\hat\mu+3.7201\hat\mu^2-5.1773\hat\mu^3+7.553\hat\mu^4-7.332\hat\mu^5),
\end{equation}
where $H$ is the entire height of the beam cross-section, and $\hat\mu=a/H$.
 
The potential energy of the rotational spring with the spring constant $k$ is 
\[
U_{crack}=\frac 12 k\a^2,
\]
where $\a$ is the  angle of twist from its equilibrium position in radians.
Since $J[y'](x)\approx\a$ for small jumps in the slope of the beam, we conclude that the total bending potential energy of the cracked beam with $m$ cracks is 
 \begin{equation}\label{phys:eq28}
U_b(y)=\frac {EI}{2}\int_0^L (y'')^2\, dx+\frac{EI}{2}\sum_{i=1}^m \frac 1{\theta_i} |J[y'](x_i)|^2.
\end{equation}

{\bf Shallow arch}. First, consider the uniform case.
Following \cite{Woinowsky1950}, the axial force $N$ in the uniform arch shown in Figure \ref{figCrArch}(a) is represented as the sum of two components $N=S_0+S_1$, where $S_0$ is the initial axial tensile force, and $S_1$ is the axial tensile force due to deflection. That is, the force is positive if it is tensile, and negative if it is compressing. The value of $S_0$ is assumed to be given, and the unknown force $S_1$ can be found through the deflection $y=y(x)$ as follows.

Let $\Delta L$ be the elongation of the arch due to the deflection. By the definition of the Young's modulus $E$
 \begin{equation}\label{phys:eq34}
S_1=EA\frac{\Delta L}{L}.
\end{equation}

The elongation is
 \begin{align}\label{phys:eq36}
\Delta L & =\int_0^L\sqrt{1+|y'(x)|^2}\, dx -L\\
& \approx\int_0^L\left(1+\frac 12|y'(x)|^2\right) dx-L=\frac 12\int_0^L|y'(x)|^2\, dx.\nonumber
\end{align}
Then
 \begin{equation}\label{phys:eq38}
S_1=\frac{EA}{2L} \int_0^L|y'(x)|^2\, dx.
\end{equation}

The potential energy $U_a$ due to the axial force $N$ is
 \begin{equation}\label{phys:eq42}
U_a(y)=\frac{EA}{2L}(\Delta L^*)^2,
\end{equation}
where $\Delta L^*=\frac{NL}{EA}$ is the elongation of the arch caused by the total axial force $N=S_0+S_1$. Therefore
\begin{equation}\label{phys:eq44}
U_a(y)=\frac{L}{2EA}(S_0+S_1)^2=
\frac{L}{2EA}\left(S_0+ \frac{EA}{2L}\int_0^L|y'(x)|^2\, dx\right)^2.
\end{equation}

Now suppose that the arch has cracks as in Section \ref{section:intro}, and in Figure \ref{figCrArch}(b).
To derive its axial potential energy $U_a(y)$, note that there exists a sequence of smooth functions $u_n$, such that $u_n\to y$ in $H^1_0(0,L)$, as $n\to\infty$. For such functions, the potential energy $U_a(u_n)$ is given by the expression in \eqref{phys:eq44}, i.e. 
\begin{equation}\label{phys:eq45}
U_a(u_n)=
\frac{L}{2EA}\left(S_0+ \frac{EA}{2L}\int_0^L|u_n'(x)|^2\, dx\right)^2, \quad n=1,2,\dots
\end{equation}
Because of the continuity of this functional on $H^1_0(0,L)$, we conclude that we can pass to the limit in \eqref{phys:eq45}, as $n\to\infty$, to get
\begin{equation}\label{phys:eq46}
U_a(y)=
\frac{L}{2EA}\left(S_0+ \frac{EA}{2L}\int_0^L|y'(x)|^2\, dx\right)^2,
\end{equation}
i.e. the same expression as \eqref{phys:eq44}, even for an arch with cracks.

{\bf Non-dimensional variables}. Now we find the non-dimensional equivalents for the above physical quantities.
Define the $t$-scale $\omega_0$, and the radius of gyration $r$ by
\begin{equation}\label{phys:eq52}
\omega_0=\left(\frac \pi L \right)^2\sqrt{\frac{EI}{\rho A}},\quad r=\sqrt{\frac{I}{A}}. 
\end{equation}
Then make the change of variables
\begin{equation}\label{phys:eq56}
x\leftarrow\frac{\pi x}{L}, \quad y\leftarrow\frac{y}{r},\quad p\leftarrow\frac{p}{EIr}\left(\frac L\pi\right)^4,\quad 
t\leftarrow \omega_0 t,\quad c_d\leftarrow \frac{c_d}{\rho A\omega_0}.
\end{equation}

Temporarily distinguish the notation for the original physical variables by assigning them the $o$ subscript, and their non-dimensional equivalents by assigning them the $n$ subscript. For example, \eqref{phys:eq56} transformation for $t$
says that $t_n=\omega_0 t_o$. Then
\begin{equation}\label{phys:eq60}
y_o(x_o,t_o)=r y_n(x_n,t_n)=ry_n\left(\frac{\pi x_o}{L},\omega_0 t_o  \right).
\end{equation}
Therefore
\begin{equation}\label{phys:eq62}
y_o'(x_o,t_o)=r \frac{\pi}{L}y_n'(x_n,t_n),\quad y_o''(x_o,t_o)=r \left(\frac{\pi}{L}\right)^2 y_n''(x_n,t_n),
\end{equation}
and
\begin{equation}\label{phys:eq64}
\dot y_o(x_o,t_o)=r \omega_0 \dot y_n(x_n,t_n),\quad \ddot y_o(x_o,t_o)=r \omega_0^2 \ddot y_n(x_n,t_n).
\end{equation}
Let the energy normalization factor $e_I$ be defined by
\begin{equation}\label{phys:eq68}
e_I=\frac L\pi\rho Ar^2\omega_0^2=\frac L\pi\rho I\omega_0^2=\left(\frac{\pi}{L}\right)^3 r^2 EI.
\end{equation}
Then the kinetic energy $T_{k,o}$ from \eqref{phys:eq4} is transformed according to
\begin{align}\label{phys:eq70}
T_{k,o}(y_o)&=\frac{\rho A}2\int_0^L|\dot y_o(x_o,t_o)|^2\,dx_o=\frac L\pi\rho Ar^2\omega_0^2\frac 12\int_0^\pi |\dot y_n(x_n,t_n)|^2\, dx_n\\
&=e_I\frac 12|\dot y_n|_H^2=e_I T_{k,n}(y_n).\nonumber
\end{align}

Consistent with the definition of the flexibility $\theta$ in \eqref{phys:eq22}, we use the transformation $\theta\leftarrow \pi \theta/L$. Therefore, for the  bending potential energy \eqref{phys:eq28} we have
\begin{align}\label{phys:eq72}
&U_{b,o}(y_o)=\frac {EI}{2}\int_0^L |y''_o(x_o)|^2\, dx_o+\frac{EI}{2}\sum_{i=1}^m \frac 1{\theta_{i,o}} |J[y'_o](x_{i,o})|^2\\
&=\left(\frac{\pi}{L}\right)^3 r^2 EI\left[\frac 12 \int_0^\pi |y''_n(x_n)|^2\, dx_n +  \frac 12\sum_{i=1}^m \frac 1{\theta_{i,n}} |J[y'_n](x_{i,n})|^2  \right]\nonumber\\ &=e_I U_{b,n}(y_n).\nonumber
\end{align}

The potential energy $U_a$ due to the axial force is given by \eqref{phys:eq46}. After the substitution \eqref{phys:eq56} we get
\begin{align*}
&U_{a,o}(y_o)=\frac{L}{2EA}\left(S_0+ \frac{EA}{2L} r^2\frac{\pi}{L}\int_0^\pi|y'_n(x_n)|^2\, dx_n\right)^2\\
&=\frac{L}{2EA}\frac{(EA)^2\pi^2r^2 r^2}{L^4}\left(\frac{L^2 S_0}{EA\pi r^2} +\frac 12\int_0^\pi|y'_n(x_n)|^2\, dx_n\right)^2\\
&=\frac{e_I}{2\pi}\left(\b+\frac 12\int_0^\pi|y'_n(x_n)|^2\, dx_n\right)^2,
\end{align*}
where the non-dimensional $\b\in\R$ is a renormalization of the axial force $S_0$.

Similarly, we conclude that $\mathcal{W}_{ext}$, and $\mathcal{W}_{d}$ are transformed by \eqref{phys:eq56}  into their non-dimensional equivalents in the same way:
\begin{align}\label{phys:eq82}
T_k\leftarrow \frac{T_k}{e_I},\quad U_a\leftarrow \frac{U_a}{e_I},\quad U_b\leftarrow \frac{U_b}{e_I},\quad \mathcal{W}_{ext}\leftarrow \frac{\mathcal{W}_{ext}}{e_I},\quad \mathcal{W}_{d}\leftarrow \frac{\mathcal{W}_{d}}{e_I}.
\end{align}

Therefore, the expressions for the non-dimensional quantities are
\begin{equation}\label{phys:eq86}
T_k(y)=\frac 12|\dot y|^2_H,\quad \mathcal{W}_{ext}(y)=(p,y)_H,\quad \mathcal{W}_{d}(y)=-c_d (\dot y,y)_H,
\end{equation}
\begin{equation}\label{phys:eq88}
U_b(y)=\frac 12\left(|y''|^2_H+\sum_{i=1}^{m} \frac 1\theta_i|J[y'](x_i)|^2 \right),
\end{equation}
\begin{equation}\label{phys:eq90}
U_a(y)=\frac{1}{2\pi}\left(\b+\frac 12|y'|^2_H\right)^2,
\end{equation}
where $\b\in \R$. 

{\bf Natural beam frequencies}.
Equation of harmonic transverse oscillations $v=v(x)$ of a uniform beam defined on interval $(0,L)$ is
\begin{equation}\label{phys:eq100}
EIv''''(x)=\omega^2\rho A v(x),\quad 0<x<L.
\end{equation}
For a cracked beam, equation \eqref{phys:eq100} is satisfied on every subinterval $(x_{i-1}, x_i)$, $i=1,\dots, m+1$. 

Using the transformations  to the non-dimensional variables \eqref{phys:eq52}--\eqref{phys:eq64}, we get
\[
EI r\left(\frac{\pi}{L}\right)^4 v_n''''(x_n, t_n)=\omega^2\rho A r v_n(x_n, t_n),
\]
where $v_n(x_n, t_n)$ is $v(x,t)$ in the new (non-dimensional) variables. 
Dropping the subscript $n$, we obtain the equation for $v$ in the non-dimensional quotients
\[
v''''=\omega^2\left(\frac{L}{\pi}\right)^4\frac{\rho A}{EI}v.
\]
Comparing this equation with the definition of the eigenvalues and the eigenfunctions $\varphi_k''''=\l_k^4\varphi_k$, we conclude that the natural beam frequencies are given by
\begin{equation}\label{phys:eq110}
\omega_k=\l_k^2\left(\frac{\pi}{L}\right)^2\sqrt{\frac{EI}{\rho A}}, \quad k\geq 1.
\end{equation}

\setcounter{equation}{0}
\section{Convex functions and subdifferentials}\label{section:conv}

Subdifferentials provide the proper mathematical framework for the abstract formulation of equations of motion.
Following \cite[Section 1.2]{Barbu2010}, let $X$ be a Hilbert space. A function $\phi : X\to (-\infty,+\infty]$ is called \emph{proper and convex} on $X$, if $\phi$ is not identically $+\infty$, and
\begin{equation}\label{conv:eq2}
\phi((1-\l)x+\l y)\leq(1-\l)\phi(x)+\l\phi(y),
\end{equation}
for any $x,y\in X$, and $\l\in[0,1]$.
The function $\phi$ is called \emph{lower-semicontinuous} on $X$, if every level set $\{x\in X : \phi(x)\leq c\},\, c>-\infty$, is closed in $X$.

Given a proper, convex, lower-semicontinuous function $\phi$ on $X$, the \emph{subdifferential} $\pa\phi : X\to X'$ is defined by
\begin{equation}\label{conv:eq6}
\pa\phi(x)=\{x^*\in X'\ :\ \phi(y)\geq \phi(x)+\langle x^*, y-x\rangle\},
\end{equation}
for any $y\in X$. Thus $\pa\phi\subset X\times X'$.

In general, $\pa\phi$ does not have to be defined everywhere on $X$. Furthermore, the mapping $x\to\pa\phi(x)\subset X'$ can be multi-valued. Geometrically, if $\pa\phi(x)$ is single-valued at $x\in X$, then $y\to \phi(x)+\langle \pa\phi(x), y-x\rangle $ is the tangent plane to the graph of $\phi$ at $x$.
%

Let $D(\phi)=\{x\in X : \phi(x)<\infty\}$, and $D(\pa\phi)=\{ x\in X : \pa\phi(x)\not=\emptyset\}$. Note that if $x\in D(\pa\phi)$, then $x\in D(\phi)$. Indeed, $\phi$ is a proper function. Therefore $D(\phi)\not=\emptyset$. Choose $y\in D(\phi)$, and $x^*\in \pa\phi(x)$. Then $\phi(x)\leq \phi(y)-\langle x^*, y-x\rangle<\infty$, as claimed. 

Let $f : X\to \R$. The \emph{directional derivative} $f'(x;y)$ of $f$ at $x \in X$ in the direction $y\in X$ is defined by
\begin{equation}\label{conv:eq10}
f'(x;y)=\lim_{\a\to 0+}\frac{f(x+\a y)-f(x)}{\a}.
\end{equation}

A function $f$ is said to be G\^ateaux differentiable at $x\in X$, if there exists $(\nabla f)(x)\in X'$, such that 
\begin{equation}\label{conv:eq12}
f'(x;y)=\langle \nabla f(x), y\rangle,
\end{equation}
for any $y\in X$. The linear functional $(\nabla f)(x)$ is called the G\^ateaux derivative of $f$ at $x\in X$.

If the convergence in \eqref{conv:eq10} is uniform in $y$ on bounded subsets, then $f$ is said to be \emph{Fr\'echet differentiable}.

The G\^ateaux differentiability of $f$ is a more restrictive condition, than $f$ is having a subdifferential. The following lemma is proved in \cite[Section 1.2]{Barbu2010}: 
\begin{lemma}\label{conv:lemma2}
If $\phi$ is convex and G\^ateaux differentiable at $x\in X$, then $\pa\phi(x)=\nabla \phi(x)$, and $\pa\phi(x)$ is a singleton.
\end{lemma}

Recall that an operator $A : X\to X'$ is called symmetric, if $\langle Au,v\rangle=\langle Av, u\rangle$, for any $u,v\in D(A)$.  
\begin{theorem}\label{conv:thm3}
Let $X$ be a Hilbert space, and $A : X\to X'$ be a linear, continuous, and symmetric operator, such that $D(A)=X$, and $\langle Au, u\rangle\geq 0$ for any $u\in X$.
Then function $\phi : X\to \R$ defined by
\begin{equation}\label{conv:thm3:eq7}
\phi(u)=\frac 12\langle Au,u \rangle,\quad u\in X,
\end{equation}
is convex, proper, and lower-semicontinuous on $X$. Moreover, it is Fr\'echet differentiable on $X$ with $\nabla\phi(u)=\pa\phi(u)=Au$ for any $u\in X$, and $D(\phi)=D(\pa\phi)=X$.
\end{theorem}
\begin{proof}
To see that $\p$ is convex, let $0\leq \l\leq 1$, and $u,v\in X$. Then
\begin{align*}
& \phi(\l u+(1-\l)v)=\frac 12 \left[\l^2\langle Au,u\rangle+2\l(1-\l)\langle Au,v\rangle+(1-\l)^2\langle Av,v\rangle\right]\\
&=  \frac 12\left[\l \langle Au,u\rangle+(1-\l)\langle Av,v\rangle-\l(1-\l)\langle A(u-v),(u-v)\rangle  \right]\\
&\leq \l\phi(u)+(1-\l)\phi(v).
\end{align*}
Since $A$ is continuous, function $\phi$ is continuous on $X$. In particular, it is lower-semicontinuous on $X$. From the continuity of $A$, we have $\|Av\|_{X'}\leq C\|v\|_X$, for some $C>0$. Since
\[
\phi(v)-\phi(u)-\langle Au,v-u\rangle=\frac 12\langle  A(v-u),v-u\rangle,
\]
we conclude that
\begin{equation}\label{conv:thm3:eq8}
|\phi(v)-\phi(u)-\langle Au,v-u\rangle|\leq \frac C2\|v-u\|_X^2.
\end{equation}
Therefore $Au=(\nabla \phi)(u)$ is the G\^ateaux (even Fr\'echet) derivative of $\phi$ at $u\in X$, and $\pa\phi=A$ is its single-valued subdifferential, with $D(\pa\phi)=D(\p)=X$.
\end{proof}

\begin{example}\label{conv:exm2}
Let $V$ be the Hilbert space defined in \eqref{hilbert:eq22}, and $\A$ be the linear operator defined in \eqref{var:eq8}. Let
\begin{equation}\label{conv:exm2:eq2}
\p(u)=\frac 12\langle \A u, u\rangle,\quad u\in V.
\end{equation}

By Theorem \ref{conv:thm3} with $X=V$, function $\p : V\to\R$ is proper, convex, and lower-semicontinuos on $V$. Furthermore, $D(\pa\p)=V$, and $\pa\p(u)=\A u$, for any $u\in V$.

For a general $u\in V$, the expression for $\pa\p(u)\in V'$ is complicated. However, if we assume that $u$ is somewhat more regular, then we can get a simpler expression for it. 

Suppose that $u\in D(\A)\subset V$. Then, by  Theorem \ref{var:thm2}, we have $\A u=u''''$ a.e. on every interval $l_i$, $i=1,\dots, m+1$. Thus, we can say that $\pa\p(u)=u''''$ a.e. on every such interval.

Suppose further, that $u\in H_0^1(0,\pi)\cap H^4(0,\pi)\subset D(\A)$. Then $u''''\in L^2(0,\pi)$, and
 $u'$ is smooth. Thus $J[u'](x_i)=0$ for any $i=1,\dots,m$, and we have
\begin{equation}\label{conv:exm2:eq4}
\langle \pa\p(u),v\rangle=\langle \A u, v\rangle=\int_0^\pi u''(x)v''(x)\, dx=\int_0^\pi u''''(x) v(x)\, dx,
\end{equation}
for any $v\in V$. Therefore, in this case $\pa\p(u)=u''''$ a.e. on $[0,\pi]$.
\end{example}

\begin{example}\label{conv:exm4}
Let $H_0^1=H_0^1(0,\pi)$ be the Hilbert space defined in \eqref{hilbert:eq60}, and $\langle\cdot, \cdot\rangle_1$ be the duality pairing between $H_0^1$ and $(H_0^1)'$.
 Let $\B$ be the linear operator defined by
\begin{equation}\label{conv:exm4:eq2}
\langle \B u, v\rangle_1=(u',v')_H,\quad  u,v\in H_0^1.
\end{equation}
Then $\B : H_0^1\to (H_0^1)'$ is  continuous, symmetric and coercive on $H_0^1$. In particular, $\B$ is positive, and its range is $(H_0^1)'$. 

Let 
\begin{equation}\label{conv:exm4:eq6}
\psi(u)=\frac 12\langle \B u, u\rangle,\quad u\in H_0^1.
\end{equation}
Theorem \ref{conv:thm3} is applicable with $X=H_0^1$, and $A=\B$. We conclude that
the function $\psi : H_0^1\to\R$ is proper, convex, and lower-semicontinuos on $H_0^1$. Furthermore, $D(\pa\psi)=H_0^1$, and $\pa\psi(u)=\B u\in (H_0^1)'$, for any $u\in H_0^1$.

 As in Example \ref{conv:exm2}, a simpler expression for the subdifferential $\pa\psi(u)$ can be obtained assuming an additional regularity of $u\in H_0^1$.

Suppose that $u\in V\subset H_0^1$. Then we have
\begin{align}\label{conv:exm4:eq8}
\langle \B u, v\rangle_1&=(u',v')_H=\int_0^\pi u'(x)v'(x)\, dx=\\
&-\sum_{i=1}^m J[u'(x_i)]v(x_i)-\sum_{i=1}^{m+1} (u'',v)_i,\nonumber
\end{align}
for any $v\in H_0^1$. Therefore, in this case,
\begin{equation}\label{conv:exm4:eq10}
\pa\psi(u)=\B u=-\sum_{i=1}^m J[u'](x_i)\d(x-x_i)-u'',
\end{equation}
where $\d(x-a),\, a\in [0,\pi]$ is the element of $(H_0^1)'$, defined by $\langle \d(x-a), v\rangle_1=v(a)$, for any $v\in H_0^1$.

Suppose further, that $u\in D(\A)\subset V$. Then, by Theorem \ref{var:thm2}, $u$ satisfies conditions \eqref{var:eq2}--\eqref{var:eq4}. In particular, $J[u'](x_i)=\theta_i u''(x_i)$. Thus, with this additional assumption on $u$, we have
\begin{equation}\label{conv:exm4:eq12}
\pa\psi(u)=\B u=-\sum_{i=1}^m \theta_i u''(x_i)\d(x-x_i)-u'',
\end{equation}
which is still an element of $(H_0^1)'$.
\end{example}

\section{Extended Hamilton's principle}\label{section:ham}
\setcounter{equation}{0}

To derive the governing equations for beams and  arches, we  use the Extended Hamilton's Principle, which accommodates non-conservative forces, see \cite{KIM20133418}.

The Principle states that the motion $y=y(x,t)$ of the system gives a stationary value to the \emph{action} of the system, i.e. to the integral
\begin{equation}\label{ham:eq2}
I(y)=\int_{t_1}^{t_2} (\mathcal{L}+\mathcal{W}_{nc})\, dt,
\end{equation}
where the Lagrangian $\mathcal{L}=T_k-U$ is the difference between the kinetic energy $T_k$ and the potential energy $U=U_b+U_a$. The term $\mathcal{W}_{nc}=\mathcal{W}_{d}+\mathcal{W}_{ext}$ for the non-conservative external work $\mathcal{W}_{nc}$ is the sum of the external work $\mathcal{W}_{ext}$, due to the load $p$, and the dissipative work $\mathcal{W}_{d}$, due to a damping force $\mathcal{F}_d$.

Thus \eqref{ham:eq2} becomes
\begin{equation}\label{ham:eq12}
I(y)=\int_{t_1}^{t_2} \left(T_k(y)-U_b(y)-U_a(y)+\mathcal{W}_{d}(y,\dot y)+\mathcal{W}_{ext}(y)  \right)\, dt,
\end{equation}
where the components of the action are given in \eqref{phys:eq86}--\eqref{phys:eq90}.

The traditional way of expressing the fact that the motion of the system gives a stationary value to the  functional $I$ defined in \eqref{ham:eq12}, is to say that its \emph{variation}  $\d I=0$. Let us examine this statement in more detail within the framework of the Hilbert spaces.

Recall that the Hilbert spaces $V, H_0^1$, and $H$ were defined in Section \ref{section:hilbert}.
Define  
\[
W=\{y\ :\ y\in L^2(0,T;V),\ \dot y\in L^2(0,T;H),\ \ddot y\in L^2(0,T;V')\},
\]
where the derivatives are taken in the sense of distributions, \cite[Chapter 2]{temam1997infinite}. This is a Hilbert space with the standard inner product and the norm.

Given $\eta\in W$, the directional derivative $I'(y,\eta)$ is defined by
\begin{equation}\label{ham:eq16}
I'(y,\eta)=\lim_{\a\to 0+}\frac{I(y+\a\eta)-I(y)}{\a},
\end{equation}
which is consistent with \eqref{conv:eq10}.

The G\^ateaux derivative $\nabla I(y)$ (if exists) is an element of the dual $W'$, such that
\begin{equation}\label{ham:eq18}
I'(y,\eta)=\langle\nabla I(y),\eta\rangle_W,
\end{equation}
for any $\eta\in W$, see  Section \ref{section:conv}, and \cite[Section 1.2]{Barbu2010}.

Functional $I$ has a stationary value at $y$ means that $\nabla I(y)=0$, which gives the meaning to the infinitesimal variation equation $\d I=0$.

To derive the governing equation, one has to transform the stationary value equation $\nabla I(y)=0$ into a more explicit form by computing the directional derivatives of $I$ along the specially chosen $\eta\in W$.

Let $\eta(x,t)=\zeta(t) v(x)$, where $\zeta : [t_1,t_2]\to\R$ is a smooth function  satisfying $\zeta(t_1)=\zeta(t_2)=0$, and $v\in V$. 

To obtain $\nabla I(y)$, we compute the corresponding directional derivatives for every term in \eqref{ham:eq12} for the same $\eta=\zeta v$. First, notice that
\begin{equation}\label{ham:eq24}
T_k(y+\a\eta)-T_k(y)= \a(\dot y, v)_H\dot\zeta  +\frac{\a^2}{2} |\dot\zeta|^2\,|v|_H^2.
\end{equation}
By the definition of the directional derivative, $T_k'(y,\eta)= \int_{t_1}^{t_2}(\dot y, v)_H\dot\zeta\, dt$.
Integrating by parts in $t$, and using $\zeta(t_1)=\zeta(t_2)=0$, we get
\begin{equation}\label{ham:eq26}
T_k'(y,\eta)= \int_{t_1}^{t_2} (\dot y(t), v)_H\dot\zeta(t)\, dt=-\int_{t_1}^{t_2} \langle\ddot y(t), v\rangle_V\zeta(t)\, dt.
\end{equation}
Similarly,
\begin{equation}\label{ham:eq28}
\mathcal{W}_{ext}'(y,\eta)= \int_{t_1}^{t_2} (p(t),v)_H\zeta(t)\, dt,
\end{equation}
where the dependency of $p=p(x,t)$ on $x$ is suppressed. 

The work $\mathcal{W}_d$ done by the non-conservative force $\mathcal{F}_d$ is not covered by the classical Hamilton's Principle. To accommodate such a case, the Principle is extended by introducing the Rayleigh's dissipation functional, for which the action is \emph{defined} in a non-variational manner, see \cite{KIM20133418} for details. In our case it implies that
\begin{equation}\label{ham:eq30}
\mathcal{W}_d'(y,\eta)= -c_d\int_{t_1}^{t_2} (\dot y(t),v)_H\zeta(t)\, dt.
\end{equation}

The expressions for the potential energies $U_b$ and $U_a$ depend on whether the system is a beam or an arch, and on the presence of cracks. They are given in \eqref{phys:eq88}--\eqref{phys:eq90}.
In every case,  $U_b(u)$ is a convex, lower-semicontinuous function on $V$,
and $U_a(u)$ is a convex, lower-semicontinuous function on $H_0^1$.
 Therefore they admit subdifferentials $\pa U_b$ and $\pa U_a$, on $V$ and $H_0^1$ correspondingly, see Section \ref{section:conv}.

 If $U_b$ and $U_a$ are Gateaux differentiable, then the subdifferentials are equal to their Gateaux derivatives. Thus, we have for $\eta=\zeta v$
\begin{align}\label{ham:eq38}
U_b'(y,\eta)&=\int_{t_1}^{t_2} \langle\pa U_b(y(t)), v\rangle_V\zeta(t)\, dt,\\ U_a'(y,\eta)&=\int_{t_1}^{t_2} \langle\pa U_a(y(t)), v\rangle_1\zeta(t)\, dt.
\end{align}

The stationary value equation $\nabla I(y)=0$, for $I(y)$ given by \eqref{ham:eq12}, implies
\begin{multline}\label{ham:eq42}
0=(\nabla I(y),\eta)_W=I'(y,\eta)\\
=\int_{t_1}^{t_2}\left[( -c_d\dot y+p(t), v)_H  -\langle \ddot y+   \pa U_b(y(t)), v\rangle_V-\langle\pa U_a(y(t)), v\rangle_1   \right] \zeta(t)\, dt.
\end{multline}
Since $\zeta$ is an arbitrary smooth function on $[t_1, t_2]$, we get
\begin{equation*}
( -c_d\dot y+p(t), v)_H  -\langle \ddot y+   \pa U_b(y(t)), v\rangle_V-\langle\pa U_a(y(t)), v\rangle_1 =0,
\end{equation*}
for any $v\in V$, and almost any $t\in [0,T]$. This can further be written as
\begin{equation}\label{ham:eq44}
\langle -\ddot y- \pa U_b(y(t))-\pa U_a(y(t)) -c_d\dot y+p(t), v\rangle_V =0.
\end{equation}
 Thus \eqref{ham:eq44} implies the following main result: equation
\begin{equation}\label{ham:eq46}
\ddot y+\pa U_b(y)+\pa U_a(y)+c_d\dot y=p,
\end{equation}
 is the abstract governing equation of the system in $V'$, a.e. for $t\in[0,T]$.
Note: $\pa U_b : V\to V'$, and $\pa U_a : H_0^1\to (H_0^1)'$.

\section{Beam equation}\label{section:beameq}
\setcounter{equation}{0}
{\bf Uniform beam}. Suppose that we have a uniform beam with no cracks, as shown in Figure \ref{figCrArch}(a). 
The classical Euler-Bernoulli beam theory 
\cite{HAN1999} 
assumes that the beam motion is due mainly to its bending, while the influence of the axial force is negligible. Thus we let $U_a=0$. 

By \eqref{phys:eq88} the bending potential energy of the uniform beam is given by
 \begin{equation}\label{beameq:eq4}
U_b(y)=\frac 12|y''|^2_H.
\end{equation}

Define operator $\A :V\to V'$ by 
 \begin{equation}\label{beameq:eq6}
 \langle \A u,v\rangle=(u'', v'')_H,
 \end{equation}
for any $u,v\in V$. The operator $\A$ is linear, symmetric, and coercive on $V$.
Then
 \begin{equation}\label{beameq:eq10}
U_b(u)=\frac 12 \langle\A u, u\rangle,
\end{equation}
for any $u\in V$. 
By Theorem \ref{conv:thm3}, $U_b$ is a convex, lower-semicontinuous function on $V$. Furthermore, $\pa U_b(u)=\A u$.
From \eqref{ham:eq46}, the abstract Uniform Beam equation in $V'$ is
 \begin{equation}\label{beameq:eq14}
\ddot y+\A y+c_d\dot y=p.
\end{equation}

Assume that $u\in D(\A)$.  By Example \ref{conv:exm2}, $\A u=u''''$ for such $u$. Therefore, with this assumption, equation \eqref{beameq:eq14} can be written as
 \begin{equation}\label{beameq:eq24}
\ddot y+y''''+c_d\dot y=p,
\end{equation}
which is the classical Euler-Bernoulli form of the Beam equation for beams without cracks.
To obtain this equation in physical variables, reverse the substitutions in Section \ref{section:phys}.

{\bf Beam with cracks}.
Now suppose that the beam has cracks as in Section \ref{section:intro}, and in Figure \ref{figCrArch}(b). Continuing with the classical Euler-Bernoulli beam theory we disregard the axial force, and let $U_a=0$.

By \eqref{phys:eq88}, the bending potential energy of the cracked beam with $m$ cracks is 
 \begin{equation}\label{beameq:eq54}
U_b(y)=\frac 12\left(|y''|^2_H+\sum_{i=1}^{m} \frac 1\theta_i|J[y'](x_i)|^2 \right).
\end{equation}
Arguing as in Example \ref{conv:exm2}, we get that $U_b$ is a convex, lower-semicontinuous function on $V$. 

Define operator $\A :V\to V'$ as in equation \eqref{var:eq8}, i.e.
 \begin{equation}\label{beameq:eq56}
\langle \A u,v\rangle= \sum_{i=1}^{m+1} (u'', v'')_i+\sum_{i=1}^m \frac 1\theta_i J[u'](x_i) J[v'](x_i), 
\end{equation}
for any $u,v\in V$.
By Lemma \ref{var:lemma1}, the operator $\A$ is linear, symmetric, and coercive on $V$.
Then 
 \begin{equation}\label{beameq:eq60}
U_b(u)=\frac 12 \langle\A u, u\rangle,
\end{equation}
for any $u\in V$. 

By Theorem \ref{conv:thm3},  $\pa U_b(u)=\A u$.
Therefore the abstract equation for the beam with cracks is
\begin{equation}\label{beameq:eq64}
\ddot y+\A y+c_d\dot y=p,
\end{equation}
which is satisfied in $V'$, a.e. for $t\in [0,T]$.

Furthermore, using Example \ref{conv:exm2},
if $u\in D(\A)$, then $u$ satisfies the boundary conditions of the problem, i.e.  \eqref{var:eq2} and \eqref{var:eq4}, as well as $\A u=u''''$ a.e. on every interval $l_i$, $i=1,\dots, m+1$.

Then equation \eqref{beameq:eq64} can be written as 
 \begin{equation}\label{beameq:eq74}
\ddot y+y''''+c_d\dot y=p.
\end{equation}
We can call it the classical Beam equation for beams with cracks. While this equation looks the same as the Beam equation for beams with no cracks \eqref{beameq:eq24}, its abstract formulation \eqref{beameq:eq64} uses the operator $\A$ defined by \eqref{beameq:eq56}, rather than by \eqref{beameq:eq6}. In particular, $y(\cdot, t)\in V$, a.e. $t\in [0,T]$.

{\bf Strong damping}. Viscous effects on the beam and arch motion are discussed in \cite{BALL1973399, Emmrich_2011}. Considerations based on the Voigt model for viscoelasticity result in the additional term $\mu\A\dot y$ in the governing equations. Here $\mu >0$ is a non-dimensional normalized dynamic viscosity coefficient. 

If such a term is present, we refer to the model as having the \emph{strong damping}. Otherwise, if $\mu=0$, the model is for the \emph{weak damping}. In particular, equations \eqref{beameq:eq64} and \eqref{beameq:eq74} describe the weak beam damping  motion case. The corresponding non-dimensional abstract and classical equations in the presence of the strong damping $\mu>0$ are 
 
 \begin{equation}\label{beameq:eq82}
\ddot y+\A y+\mu\A\dot y+c_d\dot y=p,
\end{equation}
and
 \begin{equation}\label{beameq:eq84}
\ddot y+y''''+\mu\dot y''''+c_d\dot y=p.
\end{equation}
 
\section{Arch equation}\label{section:archeq}
\setcounter{equation}{0}

{\bf Uniform shallow arch}. The potential energy for the arch contains the term $U_a$, due to the axial force. By \eqref{phys:eq90}
 \begin{equation}\label{archeq:eq4}
U_a(y)=\frac{1}{2\pi}\left(\b+\frac 12|y'|^2_H\right)^2.
\end{equation}

To compute the subdifferential of $U_a$, note that for $u\in V$,
 \begin{equation}\label{archeq:eq20}
\pa U_a(u)=\frac 1\pi\left(\b+\frac 12|y'|^2_H\right)\pa \psi(u),
\end{equation}
where $\psi$ is 
 \begin{equation}\label{archeq:eq22}
\psi(u)=\frac 12|u'|^2_H=\frac 12\langle\B u,u\rangle, \quad u\in V,
\end{equation}
see Example \ref{conv:exm4}. Since there are no cracks in the arch, we have $\pa\psi(u)=-u''$, per \eqref{conv:exm4:eq10}.

Using the subdifferential $\pa U_b(u)=\A u$, for $\A$ defined by \eqref{beameq:eq6}, the abstract form of the uniform shallow arch equation is
 \begin{equation}\label{archeq:eq24}
\ddot y+\A y-\frac 1\pi\left(\b+\frac 12|y'|^2_H\right)y''+ c_d\dot y=p,
\end{equation}
in $V'$, a.e. on $[0,T]$.

Its classical form is
 \begin{equation}\label{archeq:eq28}
\ddot y+ y''''-\frac 1\pi\left(\b+\frac 12\int_0^\pi |y'(x,t)|^2\, dx\right)y''+ c_d\dot y=p,
\end{equation}
cf. \cite[eq. (6)]{Woinowsky1950}.

{\bf Arch with cracks}.
Now suppose that the arch has cracks as in Section \ref{section:intro}, and in Figure \ref{figCrArch}(b).
Its axial potential energy $U_a(y)$ has the same expression as \eqref{archeq:eq4}, even if the arch has cracks. 

However, the subdifferential $\pa U_a(u)$ is different from the one in the smooth case. It has been computed in Example \ref{conv:exm4} as
 \begin{equation}\label{archeq:eq40}
\pa\psi(u)=\B u=-\sum_{i=1}^m J[u'](x_i)\d(x-x_i)-u'',
\end{equation}
for $u\in V$. The bending potential $U_b$ is given by \eqref{beameq:eq54}. Then its subdifferential $\pa U_b(u)=\A u$, and we get from  \eqref{ham:eq46} that
 \begin{equation}\label{archeq:eq44}
\ddot y+\A y-\frac 1\pi\left(\b+\frac 12|y'|^2_H\right)\left(\sum_{i=1}^m J[y'](x_i)\d(x-x_i)-y''  \right) + c_d\dot y=p,
\end{equation}
is the abstract equation for a shallow arch with cracks in $V'$, a.e. $t\in[0,T]$.

Then, assuming that the function $y$ is smooth as discussed in Example \ref{conv:exm4}, we can use \eqref{conv:exm4:eq12} for the subdifferential $\pa U_a(u)$, and $\pa U_b(u)=\A u=u''''$. This results in  
\begin{align}\label{archeq:eq48}
&\ddot y+y''''\\
&-\frac 1\pi\left(\b+\frac 12|y'|^2_H\right)\left(\sum_{i=1}^m \theta_i y''(x_i,t)\d(x-x_i)-y''\right) + c_d\dot y=p,\nonumber
\end{align}
which can be called the "classical" form of the shallow arch equation with cracks. These equations are also referred to as describing the weak arch damping motion.

{\bf Strong damping}.  Equations \eqref{archeq:eq44} and \eqref{archeq:eq48} describe the weak arch damping $\mu=0$  motion case. The corresponding non-dimensional abstract and "classical" equations in the presence of the strong damping $\mu>0$ are
 \begin{equation}\label{archeq:eq54}
\ddot y+\A y+\mu\A\dot y-\frac 1\pi\left(\b+\frac 12|y'|^2_H\right)\left(\sum_{i=1}^m J[y'](x_i)\d(x-x_i)-y''  \right) + c_d\dot y=p,
\end{equation}
and
\begin{align}\label{archeq:eq58}
&\ddot y+y''''+\mu \dot y''''\\
&-\frac 1\pi\left(\b+\frac 12|y'|^2_H\right)\left(\sum_{i=1}^m \theta_i y''(x_i,t)\d(x-x_i)-y''\right) + c_d\dot y=p,\nonumber
\end{align}
see Section \ref{section:beameq}.

\bibliographystyle{plain} 
\bibliography{references}

\end{document}